\documentclass[final, nomarks]{dmtcs-episciences}
\usepackage[backend=bibtex,bibencoding=utf8,style=ieee, sorting=nty]{biblatex}
\usepackage[utf8]{inputenc}
\usepackage[T1]{fontenc}
\usepackage{mathtools}
\usepackage{permpatts}
\usetikzlibrary{decorations.markings}
\tikzset{->-/.style={decoration={
  markings,
  mark=at position #1 with {\arrow{>}}},postaction={decorate}}}
\usepackage{amsmath}
\usepackage{amsthm}
\usepackage{enumerate}
\usepackage{xspace}
\usepackage{xinttools}
\usepackage{tabularx}

\theoremstyle{definition}
\newtheorem{theorem}{Theorem}[section]

\newtheorem{proposition}[theorem]{Proposition}
\newtheorem{definition}[theorem]{Definition}
\newtheorem{lemma}[theorem]{Lemma}
\numberwithin{equation}{section}
\newtheorem{example}[theorem]{Example}
\newtheorem{note}[theorem]{Note}

\newcommand{\ie}{\textit{i.e.}\xspace}

\DeclareRobustCommand{\Section}{Section\xspace}

\DeclareRobustCommand{\Figure}{Figure\xspace}

\DeclareRobustCommand{\Equation}{equation\xspace}

\DeclareRobustCommand{\Lemma}{Lemma\xspace}

\DeclareRobustCommand{\Proposition}{Proposition\xspace}
\DeclareRobustCommand{\Note}{Note\xspace}
\DeclareRobustCommand{\Example}{Example\xspace}
\DeclareRobustCommand{\Definition}{Definition\xspace}

\DeclareRobustCommand{\SectionRef}[1]{\Section~\ref{#1}}

\DeclareRobustCommand{\FigureRef}[1]{\Figure~\ref{#1}}

\DeclareRobustCommand{\EquationRef}[1]{\Equation~\eqref{#1}}

\DeclareRobustCommand{\LemmaRef}[1]{\Lemma~\ref{#1}}

\DeclareRobustCommand{\PropositionRef}[1]{\Proposition~\ref{#1}}
\DeclareRobustCommand{\NoteRef}[1]{\Note~\ref{#1}}
\DeclareRobustCommand{\ExampleRef}[1]{\Example~\ref{#1}}
\DeclareRobustCommand{\DefinitionRef}[1]{\Definition~\ref{#1}}
\makeatletter
\def\saveenum{\xdef\@savedenum{\the\c@enumi\relax}}
\def\resetenum{\global\c@enumi\@savedenum}
\makeatother
\newcolumntype{Y}{>{\centering\arraybackslash}X}
\newcounter{VarNumber}
\chardef\lastvar=68
\chardef\var=\value{VarNumber}
\newcommand{\scriptvar}{\mathcal{\var}}
\newcommand{\nextvar}[1][\lastvar]{%
\stepcounter{VarNumber}
\chardef#1=\var
\chardef\lastvar=\var
\chardef\var=\value{VarNumber}
}

\newcommand{\coinc}[2][]{
\cong_{#2}^\text{#1}
}

\newcommand{\weqv}[2][]{
\overset{\makebox(0,0){\scalebox{0.5}{\text{W}}}}{\cong}_{#2}^\text{#1}}

\newcommand{\et}[0]{
\text{ and }
}

\author{Murray Tannock \and Henning Ulfarsson}

\title{Equivalence classes of mesh patterns with a dominating pattern\thanks{Research partially supported by grant 141761-051 from the Icelandic Research Fund.}}
\affiliation{School of Computer Science, Reykjavik University, Reykjavik, Iceland}
\received{2017-4-25}
\revised{2017-10-17}
\accepted{2018-1-16}
\publicationdetails{19}{2018}{2}{6}{3283}
\keywords{permutation, pattern, mesh pattern, pattern coincidence}

\begin{document}
\maketitle
\begin{abstract}
    Two mesh patterns are coincident if they are avoided by the same set of
    permutations, and are Wilf-equivalent if they have the same number of
    avoiders of each length. We provide sufficient conditions for coincidence of
    mesh patterns, when only permutations also avoiding a longer classical
    pattern are considered. Using these conditions we completely classify
    coincidences between families containing a mesh pattern of length 2 and a
    classical pattern of length 3. Furthermore, we completely Wilf-classify mesh
    patterns of length 2 inside the class of 231-avoiding permutations.

\end{abstract}

\section{Introduction}
The study of permutation patterns began as a result of Knuth's statements on
stack sorting in \emph{The Art of Computer Programming}~\cite[p.~243,
Ex.~5,6]{Knuth:1997:ACP:260999}. The original concept---a subsequence of symbols
having a particular relative order, now known as classical patterns---has been
expanded to a variety of definitions. \textcite{babstein2000} considered
\emph{vincular} patterns (also known as \emph{generalised} or \emph{dashed}
patterns) where two adjacent entries in the pattern can be required to be
adjacent in the permutation. \textcite{MR2652101} look at classes of patterns
where entries can also be required to be consecutive in value, these are called
\emph{bivincular} patterns. \emph{Bruhat-restricted} patterns were studied by
\textcite{MR2264071} to establish necessary conditions for a Schubert variety to
be Gorenstein. These definitions are subsumed under the definition of \emph{mesh
patterns}, introduced by \textcite{journals/combinatorics/BrandenC11} to capture
explicit expansions for certain permutation statistics.

When considering permutation patterns some of the main questions posed relate to
how and when a pattern is avoided by, or contained in, an arbitrary set of
permutations. Two patterns \(\pi\) and \(\sigma\) are \emph{Wilf-equivalent} if
the number of permutations that avoid \(\pi\) of length \(n\) is equal to the
number of permutations that avoid \(\sigma\) of length \(n\). A stronger
equivalence condition is that of \emph{coincidence}, where the set of
permutations avoiding \(\pi\) is exactly equal to the set of permutations
avoiding \(\sigma\). Avoiding pairs of patterns of the same length with certain
properties has been studied, \textcite{MR2178749} considered avoiding a pair of
vincular patterns of length 3. \textcite{2015arXiv151203226B} study avoiding a
vincular and a covincular pattern simultaneously in order to achieve several
counting results. However, little work has been done on avoiding a mesh pattern
and a classical pattern simultaneously.

In this work we aim to establish some ground in this field by computing
coincidences and Wilf-classes and calculating some of the enumerations of
avoiders of a mesh pattern of length 2 and a classical pattern of length 3. We
begin by establishing coincidences between mesh patterns of length 2 while
avoiding a dominating pattern by computational methods, which are then used to
establish three ``Dominating Pattern Rules'' as well as some special cases that
can be used to calculate coincidences.

 We then use these coincidence classes to calculate Wilf-equivalence classes
 showing some of the methods used.

\section{Mesh patterns}
A \emph{permutation} is a bijection from the set \(\nrange{n} =
\setrange{1}{n}\) to itself. The set of all such bijections is denoted
\(\mathfrak{S}_n\) and has \(n!\) elements. We can denote an individual
permutation \(\pi\in\mathfrak{S}_n\) in \emph{one-line notation} by writing the
entries of the permutation in order, therefore \(\pi =
\perm{\pi(1),\pi(2),\dotsm,\pi(n)}\). The set \(\mathfrak{S}_0\) has exactly one
element, the empty permutation \(\varepsilon\).

\begin{definition}
    Two strings of integers \(\alpha_1\alpha_2\dotsm\alpha_n\) and
    \(\beta_1\beta_2\dotsm\beta_n\) are said to be \emph{order isomorphic}
    if they share the same relative order, \ie, \(\alpha_r<\alpha_s\) if and
    only if \(\beta_r<\beta_s\).
\end{definition}

The definition of order isomorphism allows us to give the meaning of containment for classical
permutation patterns.
\begin{definition}
    A permutation \(\pi\in\mathfrak{S}_n\) \emph{contains} the permutation
    \(\sigma\in\mathfrak{S}_k\) (\mbox{denoted} \(\sigma \preceq \pi\)) if there
    is some sequence \( i_1,i_2,\dotsc,i_k\) such that \(1\le
    i_1<i_2<\dotsm<i_k\le n\) and the sequence
    \(\pi(i_1)\pi(i_2)\dotsm\pi(i_k)\) is order isomorphic to
    \(\sigma(1)\sigma(2)\dotsm\sigma(k)\). If this is the case the sequence
    \(\pi(i_1)\pi(i_2)\dotsm\pi(i_k)\) is called an \emph{occurrence} of
    \(\sigma\) in \(\pi\). If \(\pi\) does not contain \(\sigma\), we say that
    \(\pi\) \emph{avoids} \(\sigma\). In this context \(\sigma\) is called a
    (\emph{classical}) \emph{permutation pattern}.
\end{definition}

\begin{example}
\label{ex:contexmpl}
The permutation \(\pi = \perm{2,4,1,5,3}\) contains the pattern \(\sigma =
\perm{2,3,1}\), since the second, fourth and fifth elements (\(453\)) are order
isomorphic to \(\perm{231}\). The permutation also contains the occurrence
\(241\) of the same pattern. The permutation \(\perm{2,4,1,5,3}\) avoids the
pattern \(\perm{3,2,1}\).
\end{example}

We denote the set of permutations of length \(n\) avoiding a pattern \(\sigma\)
as \(\Av_n(\sigma)\) and let
\(\av{\sigma}=\bigcup_{i=0}^{\infty}\Av_i(\sigma)\).

We can display a permutation graphically in a \emph{plot}, where we display the
points \(G(\pi) = \{(i,\pi(i))\mid i \in \nrange{n} \}\) in a Cartesian
coordinate system. The plots of the permutations \(\pi = \perm{2,4,1,5,3}\) and
\(\sigma = \perm{2,3,1}\) can be seen in \FigureRef{fig:plots}.
\FigureRef{fig:containment} shows the containment of \(\sigma\) in \(\pi\) as in
\ExampleRef{ex:contexmpl}.

The boxes in the plot of a permutation are denoted by \(\boks{i,j}\), where the
point \((i,j)\) is the lower left corner of the box.

\begin{figure}[htb]
    \begin{center}
    \raisebox{6ex}{\(G(\pi)=\)}
    \begin{tikzpicture}[scale=\picscale]
        \modpattern{}{2,4,1,5,3}{}
    \end{tikzpicture}
    \raisebox{6ex}{\(\quad{}G(\sigma)=\)}
    \raisebox{0.35cm}{
    \begin{tikzpicture}[scale=\picscale]
        \modpattern{}{2,3,1}{}
    \end{tikzpicture}}

        \caption{The plots of the permutations \(\pi\) and \(\sigma\).}
        \label{fig:plots}
    \end{center}
\end{figure}
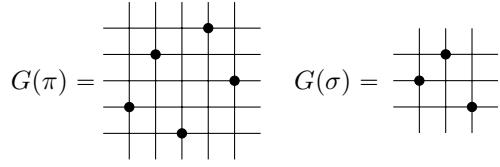

\begin{figure}[htb]
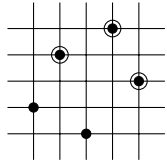

    \begin{center}
    \(\pattern{2,4,5}{2,4,1,5,3}{}\)
        \caption{The occurrence of 231 in 24153 corresponding to 453.}
        \label{fig:containment}
    \end{center}
\end{figure}

\begin{definition}
A \emph{mesh pattern} is a pair
\begin{equation*}
    p = (\tau,R)\text{ with } \tau \in \mathfrak{S}_k \text{ and } R \subseteq [0,k]\times [0,k].
\end{equation*}
\end{definition}

Formally defined by \textcite{journals/combinatorics/BrandenC11}, an
\emph{occurrence} of \(p\) in \(\pi\) is a subset \(\omega\) of the plot of
\(\pi, G(\pi) = \{(i,\pi(i)\mid i\in\nrange{n}\}\) such that there are
order-preserving injections \(\alpha,\beta:\nrange{k}\mapsto\nrange{n}\)
satisfying the following two conditions.

Firstly, \(\omega\) is an occurrence of \(\tau\) in the classical sense
  \begin{enumerate}[i.]
    \item \(\omega = \{(\alpha(i),\beta(j)):(i,j)\in G(\tau)\}\).
  \end{enumerate}\saveenum
  Define \(R_{ij} = [\alpha(i)+1,\alpha(i+1)-1]\times[\beta(j)+1,\beta(j+1)-1]\)
   for \(i,j\in[0,k]\) where \(\alpha(0)=\beta(0)=0\) and \(\alpha(k+1)=\beta(k+1)=n+1\). Then
   the second condition is
  \begin{enumerate}[i.]\resetenum
  \item if \(\boks{i,j} \in R \text{ then } R_{ij} \cap G(\pi) = \emptyset\).
\end{enumerate}
We call \(R_{ij}\) the \emph{region corresponding to} \(\boks{i,j}\).

\begin{example}
    The pattern \(p=\mperm{2,1,3}{\{\boks{0,1},\boks{0,2},\boks{1,0},\boks{1,1},\boks{2,1},\boks{2,2}\}}=
    \textpattern{}{ 2, 1, 3 }{ 0/1, 0/2, 1/0, 1/1, 2/1, 2/2 }\) is contained in
    \(\pi = \perm{3,4,2,1,5}\).
    Let us consider the plot for the permutation \(\pi\). The subsequence \(325\) is
    an occurrence of \(213\) in the classical sense and the remaining points of \(\pi\)
    are not contained in the regions corresponding to the shaded boxes in \(p\).
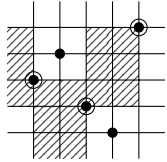
    \begin{equation*}
        \pattern{1,3,5}{3,4,2,1,5}{0/2,0/3,0/4,
                                   1/0,1/1,1/2,
                                   2/0,2/1,2/2,
                                   3/2,3/3,3/4,
                                   4/2,4/3,4/4}
    \end{equation*}
    The subsequence \(325\) is therefore an occurrence of the pattern \(p\) in \(\pi\).
\end{example}

We denote the avoidance sets for mesh patterns in the same way as for classical
patterns. Given a mesh pattern \(p=(\sigma,R)\) we say that \(\sigma\) is the
\emph{underlying classical pattern} of \(p\).

We define containment of a mesh pattern \(p\) in another mesh pattern \(q\) as
above, with the additional condition that if \(\boks{i,j}\in R\text{ then }
R_{ij}\) is contained in the mesh set of \(q\). More formally:

\begin{definition}
A mesh pattern \(q=(\kappa,T)\) \emph{contains} a mesh pattern \(p=(\tau,R)\) as
a \emph{subpattern} if \(\kappa\) contains \(p\) and
\(\left(\bigcup_{\boks{i,j}\in R} R_{ij}\right)\subseteq T\).
\end{definition}

\begin{example}
    The pattern \(p=\mperm{2,1,3}{\{\boks{0,1},\boks{0,2},\boks{1,0},\boks{2,2}\}}=
    \textpattern{}{ 2, 1, 3 }{ 0/1, 0/2, 1/0, 2/2 }\) is contained in the pattern
    \(q = \textpattern{}{4,2,3,1,5}{0/1,0/2,0/3,0/4,
                               1/0,1/1,1/2,1/3,1/4,
                               2/1,2/2,2/3,2/4,
                               3/2,3/4,
                               4/2,4/3,4/4} \) as a subpattern.
    The highlighted points form an occurrence of \(p\) in \(q\)
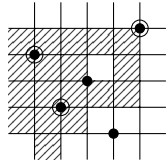
    \begin{equation*}
        \pattern{1,2,5}{4,2,3,1,5}{0/1,0/2,0/3,0/4,
                                   1/0,1/1,1/2,1/3,1/4,
                                   2/1,2/2,2/3,2/4,
                                   3/2,3/4,
                                   4/2,4/3,4/4}
    \end{equation*}
    The permutation \(\perm{4,2,3,1,5}\) also contains \(p\) in the usual sense.
\end{example}

\section{Coincidences between Mesh Patterns under a Dominating Pattern}
\label{sec:coincs}
Coincidences among small mesh patterns have been considered by
\textcite{DBLP:journals/corr/ClaessonTU14}, in which the authors use the
Simultaneous Shading Lemma, a closure result and one special case to fully
classify coincidences among mesh patterns of length 2.

Recall that two patterns \(\lambda\) and \(\gamma\) are considered
\emph{coincident} if the set of permutations that avoid \(\lambda\) is the same
as the set of permutations that avoid \(\gamma\), \ie, \(\av{\lambda} =
\av{\gamma}\). This is equivalent to \(\lambda\) and \(\gamma\) being contained
in the same set of permutations, \ie, \(\cont{\lambda} = \cont{\gamma}\).

We will consider the avoidance sets \(\av{\pi,p}\) where \(\pi\) is a fixed
classical pattern of length 3 and \(p\) is a mesh pattern of length 2 that
varies. The classical pattern \(\pi\) will be called the \emph{dominating
pattern}. Given such a dominating pattern \(\pi\) we will write \(p_1
\coinc{\pi} p_2\) if for the mesh patterns, \(p_1\) and \(p_2\),  the sets
\(\av{\pi,p_1}\) and \(\av{\pi,p_2}\) are equal. If this is the case we say
the two mesh patterns are \emph{coincident under} \(\pi\).

Our first step is to calculate whether \(\av{\pi,p_1}\) and \(\av{\pi,p_2}\) are
equal up to permutations of length 11 computationally, if they are we write
\(p_1 \coinc[comp]{\pi} p_2\). The equivalence relation \(\coinc{\pi}\) is a
refinement of the equivalence relation \(\coinc[comp]{\pi}\), and therefore the
\(\coinc{\pi}\)-equivalence classes form partitions of the
\(\coinc[comp]{\pi}\)-equivalence classes.

We then prove three propositions called ``Dominating Pattern Rules'', if the
First Dominating Pattern Rule shows that \(p_1 \coinc{\pi} p_2\) then we write
\(p_1 \coinc[(1)]{\pi} p_2\). If a combination of the First and Second
Dominating Pattern Rules show that \(p_1 \coinc{\pi} p_2\) then we write
\(p_1 \coinc[(2)]{\pi} p_2\); similarly we write \(p_1 \coinc[(3)]{\pi} p_2\) if
a combination of all three rules shows coincidence. These three equivalence
relations are successive refinements of \(\coinc{\pi}\). For future reference
we record the following:

\begin{note}
\label{note:main}
    If we find that \(\coinc[(i)]{\pi}\), for \(i = 1, 2\), or \(3\), equals
    \(\coinc[comp]{\pi}\) then the equivalence relations \(\coinc[(i)]{\pi}\),
    \(\coinc{\pi}\), and \(\coinc[comp]{\pi}\) are all equal.
\end{note}

In order to describe the rules it is useful to have a notion for inserting
points, ascents, and descents into a mesh pattern.
\begin{definition}
\label{def:ap}
Let \(p=(\tau,R)\) be a mesh pattern of length \(n\) such that
\(\boks{i,j}\notin R\). We define a mesh pattern \(p^\boks{i,j} =
(\tau^\prime,R^\prime)\) of length \(n+1\) as the pattern where a point is
\emph{inserted} into the box \(\boks{i,j}\) in \(G(p)\). Formally the new
underlying classical pattern is defined by
\begin{equation*}
\tau^\prime(k) = \begin{cases}
    j+1 & \text{if } k = i+1\\
    \tau(k) & \text{if } \tau(k)\le j \text{ and }k\le i\\
    \tau(k)+1 & \text{if } \tau(k)> j \text{ and }k\le i\\
    \tau(k-1) & \text{if } \tau(k)\le j \text{ and }k> i+1\\
    \tau(k-1)+1 & \text{if } \tau(k)> j \text{ and }k> i+1\\
\end{cases}
\end{equation*}
While the mesh becomes
\begin{equation*}
\begin{aligned}
R^\prime=&\{\boks{k,\ell}\mid k\le i, \ell\le j, \boks{k,\ell}\in R\} \cup\\
&\{\boks{k,\ell}\mid k\le i, \ell > j, \boks{k,\ell-1}\in R\}\cup\\
&\{\boks{k,\ell}\mid k > i, \ell \le j, \boks{k-1,\ell}\in R\}\cup\\
&\{\boks{k,\ell}\mid k > i, \ell > j, \boks{k-1,\ell-1}\in R\}\\
\end{aligned}
\end{equation*}
\end{definition}
In addition, we give the following definitions:
\begin{definition}
Let \(p=(\tau,R)\) be a mesh pattern of length \(n\) such that
\(\boks{i,j}\notin R\) and \(p^\boks{i,j} = (\tau^\prime,R^\prime)\) is as
defined in \DefinitionRef{def:ap}. We define the following four modifications of
\(p^\boks{i,j}\).
\begin{align*}
p^\boksup{i,j} &= (\tau^\prime, R^\prime \cup \{\boks{i,j+1},\boks{i+1,j+1}\})\\
p^\boksright{i,j} &= (\tau^\prime, R^\prime \cup \{\boks{i+1,j},\boks{i+1,j+1}\})\\
p^\boksdown{i,j} &= (\tau^\prime, R^\prime \cup \{\boks{i,j},\boks{i+1,j}\})\\
p^\boksleft{i,j} &= (\tau^\prime, R^\prime \cup \{\boks{i,j},\boks{i,j+1}\})\\
\end{align*}
\end{definition}
Informally, these are considering the topmost, rightmost, leftmost, or
bottommost point in \(\boks{i,j}\). We collect the resulting mesh patterns in a
set
\begin{equation*}
p^\boksall{i,j}=\left\{
p^\boks{i,j},
p^\boksright{i,j},
p^\boksleft{i,j},
p^\boksup{i,j},
p^\boksdown{i,j}\right\}
\end{equation*}
See \FigureRef{fig:addp} for an example of adding a point into a mesh pattern.

\begin{figure}
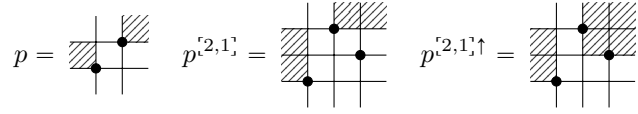

\begin{equation*}
p = \pattern{}{1,2}{0/1,2/2} \quad
p^\boks{2,1}= \pattern{}{1,3,2}{0/1,0/2,2/3,3/3}
\quad
p^\boksup{2,1}
= \pattern{}{1,3,2}{0/1,0/2,2/2,2/3,3/2,3/3}
\end{equation*}
\caption{The result of inserting a point into \(p=(12,\{\boks{0,1},\boks{2,2}\})\)}
\label{fig:addp}
\end{figure}

\begin{definition}
Let \(p=(\tau,R)\) be a mesh pattern of length \(n\) such that
\(\boks{i,j}\notin R\). We define a mesh pattern \(p^{\boks{i,j}_a} =
(\tau^\prime,R^\prime)\) (respectively, \(p^{\boks{i,j}_d}\)) of length \(n+2\) as the pattern
where an ascent (respectively descent) is \emph{inserted} into the box \(\boks{i,j}\) in
\(G(p)\). Formally the new underlying classical pattern is defined by
\begin{equation*}
\tau^\prime(k) = \begin{cases}
    j+t & \text{if } k = i+t,t\in\{1,2\}\\
    \tau(k) & \text{if } \tau(k)\le j \text{ and }k\le i\\
    \tau(k)+2 & \text{if } \tau(k)> j \text{ and }k\le i\\
    \tau(k-2) & \text{if } \tau(k)\le j \text{ and }k> i+2\\
    \tau(k-2)+2 & \text{if } \tau(k)> j \text{ and }k> i+2\\
\end{cases}
\end{equation*}
The ordering of the top branch determines whether an ascent(or descent) is added.

The mesh becomes
\begin{equation*}
\begin{aligned}
R^\prime=&\{\boks{k,\ell}\mid k\le i, \ell\le j, \boks{k,\ell}\in R\} \cup\\
&\{\boks{k,\ell}\mid k\le i, \ell > j, \boks{k,\ell-2}\in R\}\cup\\
&\{\boks{k,\ell}\mid k > i, \ell \le j, \boks{k-2,\ell}\in R\}\cup\\
&\{\boks{k,\ell}\mid k > i, \ell > j, \boks{k-2,\ell-2}\in R\}\cup\\
&\{\boks{i+1,j},\boks{i+1,j+1},\boks{i+1,j+2}\}
\end{aligned}
\end{equation*}
meaning that the ascent (or descent) is not allowed to be split in the box.
\end{definition}
An example of adding an ascent to a mesh pattern can be seen in \FigureRef{fig:adda}.
\begin{figure}
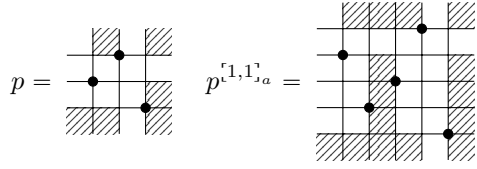

\begin{equation*}
p = \pattern{}{2,3,1}{0/0,1/0,1/3,
                      3/0,3/1,3/3} \quad
p^{\boks{1,1}_a}= \pattern{}{4,2,3,5,1}{0/0,
                        1/0,1/5,
                        2/0,2/1,2/2,2/3,2/5,
                        3/0,3/5,
                        5/0,5/1,5/2,5/3,5/5}
\end{equation*}
\caption{The result of inserting an ascent into
\(p=(231,\{\boks{0,0},\boks{1,0},\boks{1,3},\boks{3,0},\boks{3,1},\boks{3,3}\})\)}
\label{fig:adda}
\end{figure}

We now attempt to fully classify coincidences in families characterised by
avoidance of a classical pattern of length \(3\) and a mesh pattern of length
\(2\), that is finding and explaining all coincidences between mesh patterns
\(m\) and \(m^\prime\), \(m\coinc{\pi} m^\prime\), where \(\pi\) is a classical
pattern of length \(3\).

It can be easily seen that in order to classify coincidences one need only
consider coincidences within the family of mesh patterns with the same
underlying classical pattern, this is due to the fact that
\(\perm{2,1}\in\av{\mperm{1,2}{R}}\) and
\(\perm{1,2}\in\av{\{\mperm{2,1}{R}\}}\) for all mesh-sets \(R\).

We know that there are a total of \(512\) mesh-sets for each underlying
classical pattern. By use of the previous results of
\textcite{DBLP:journals/corr/ClaessonTU14} the number of coincidence classes can
be reduced to \(220\).

\subsection{Coincidence classes of Av(\{321, (21, \textit{R})\})}
Through experimentation, considering avoidance of permutations up to length
\(11\), we discover that there are \(29\) \(\coinc[comp]{321}\)-equivalence
classes where the underlying classical pattern of the mesh pattern is \(21\).

\begin{proposition}[First Dominating Pattern Rule]
    \label{prop:dom1}
    Given two mesh patterns \(m_1 =(\sigma, R_1)\) and \(m_2 = (\sigma, R_2)\),
    and a dominating classical pattern \(\pi = (\pi,\emptyset)\) such that
    \(\setsize{\pi} \le \setsize{\sigma} + 1\), then \(m_1 \coinc{\pi} m_2 \)
    if
    \begin{enumerate}
        \item The mesh set \(R_2 = R_1 \cup \{\boks{a,b}\}\)
        \item \(\pi \preceq \sigma^\boks{a,b}\)\label{prop:dom1:cont}
    \end{enumerate}
\end{proposition}

This rule can be understood in graphical form. In the pattern in
\FigureRef{fig:rule1} we can gain shading in the boxes \(\boks{0,2}\), \(\boks{2,0}\)
since if there is a point in either of these boxes there would be an occurrence of
the dominating pattern \(\perm{3,2,1}\).

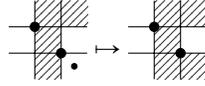
\begin{figure}[htb]
    \begin{center}
    \begin{tikzpicture}[scale=\picscale]
        \filldraw (2.5,0.5) circle (3pt);
        \modpattern{}{2,1}{1/0,1/1,1/2,2/2}
    \end{tikzpicture}
        \raisebox{2ex}{\(\mapsto\)}
    \begin{tikzpicture}[scale=\picscale]
        \modpattern{}{2,1}{1/0,1/1,1/2,2/2,2/0}
    \end{tikzpicture}

        \caption{Visual depiction of first dominating pattern rule.}
        \label{fig:rule1}
    \end{center}
\end{figure}

In order to prove the proposition we must first make the following note.

\begin{note}
    \label{not:downcmesh}
    Let \(R_1 \subseteq R_2\). Then any occurrence of \((\tau, R_2)\) in a permutation
    is an occurrence of \((\tau,R_1)\).
\end{note}

\begin{proof}[Proof of \PropositionRef{prop:dom1}]
    Since \(R_1\) is a subset of \(R_2\), \NoteRef{not:downcmesh} implies that
    \(\av{\{\pi,m_1\}} \subseteq \av{\{\pi,m_2\}}\).

    Now we consider a permutation \(\omega \in \av{\pi}\), containing an
    occurrence of \(m_1\). If there is a point in the region corresponding to
    the box \(\boks{a,b}\), then that point, along with the points of the
    occurrence of \(m_1\), form an occurrence of \(\sigma^\boks{a,b}\). Then
    condition~\eqref{prop:dom1:cont} of the proposition implies an occurrence of
    \(\pi\). Therefore there can be no points in this region, which implies the
    occurrence of \(m_1\) is an occurrence of \(m_2\). Hence every occurrence of
    \(m_1\) is in fact an occurrence of \(m_2\), and we have that
    \(\av{\{\pi,m_2\}} \subseteq \av{\{\pi,m_1\}}\).

    Taking both directions of the containment we can therefore draw the
    conclusion that \(m_1 \coinc{\pi} m_2\).
\end{proof}

After implementing the First Dominating Pattern Rule we find that there are
\(29\) \(\coinc[(1)]{321}\)-equivalence classes, of mesh patterns where the
underlying classical pattern is \(21\). By \NoteRef{note:main} there are
therefore exactly \(29\) \(\coinc{321}\)-equivalence classes.

\subsection{Coincidence classes of Av(\{231, (21, \textit{R})\}).}
By application of First Dominating Pattern Rule we obtain \(43\)
\(\coinc[(1)]{231}\)-equivalence classes between mesh patterns with \(21\) as an
underlying classical pattern. Experimentation shows that there are \(39
\coinc[comp]{231}\)-equivalence classes between mesh patterns with \(21\) as an
underlying classical pattern, for example the following two patterns are in the
same \(\coinc[comp]{231}\)-equivalence class, but this is not explained by the
First Dominating Pattern Rule.
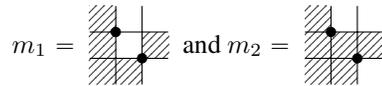
\begin{equation*}
    m_1 = \pattern{}{2,1}{0/0,0/1,0/2,1/0,2/1} \text{ and } m_2 = \pattern{}{2,1}{0/0,0/1,0/2,1/0,1/1,2/1}
\end{equation*}

Consider an occurrence of \(m_1\) in a permutation in \(\av{\perm{2,3,1}}\),
consisting of elements \(x\) and \(y\). If the region corresponding to the box
\(\boks{1,1}\) is empty we have an occurrence of \(m_2\). Otherwise, if there is
any ascent in this box then we would have an occurrence of \(\perm{2,3,1}\),
however, since the permutation is in \(\av{\perm{2,3,1}}\) this is not possible.
This box must therefore contain a (non-empty) decreasing subsequence. This gives
rise to the following lemma:

\begin{lemma}
    \label{lem:incdec}
    Let \(m =(\sigma, R)\) be a mesh pattern, where \(\boks{a,b} \notin R\),
    and \(\pi = (\pi,\emptyset)\) be a dominating classical pattern. If
    \(\pi \preceq m^{\boks{a,b}_a}\)\\(\(\pi \preceq m^{\boks{a,b}_d}\)),
    then in any occurrence of \(m\) in a permutation \(\varrho\), the region
    corresponding to the box \(\boks{a,b}\) can only contain an decreasing
    (increasing) subsequence of \(\varrho\).
\end{lemma}
The proof is analogous to the proof of \PropositionRef{prop:dom1}.

Going back to our example mesh patterns
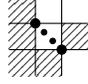
\begin{equation*}
    \begin{tikzpicture}[scale=\picscale]
        \modpattern{}{2,1}{0/0,0/1,0/2,1/0,2/1}
        \filldraw (1.66,1.33) circle (3pt);
        \filldraw (1.33,1.66) circle (3pt);
    \end{tikzpicture}
\end{equation*}
we know that the region corresponding to the box \(\boks{1,1}\) contains a
decreasing subsequence. If we let \(z\) be the topmost point in this decreasing
subsequence, then \(xz\) is an occurrence of \(m_2\). This shows that our two
example patterns are coincident. This example generalises into the following
rule.

\begin{proposition}[Second Dominating Pattern Rule]
    \label{prop:dom2}
    Given two mesh patterns \(m_1 =(\sigma, R_1)\) and \(m_2 = (\sigma, R_2)\),
    and a dominating classical pattern \(\pi = (\pi,\emptyset)\) such that
    \(\setsize{\pi} \le \setsize{\sigma} + 2\), then \(m_1 \coinc{\pi} m_2 \)
    if
\begin{enumerate}
  \item The mesh set \(R_2 = R_1 \cup \{\boks{a,b}\}\)
  \item\label{prop:dom2:cond} Any one of the following four conditions hold
    \begin{enumerate}
    \item\label{prop:dom2:condc} \(\pi \preceq \sigma^{\boks{a,b}_a}\) and
      \begin{enumerate}
        \item \((a+1,b) \in \sigma\) and \(\boks{a+1,b-1}\notin R_1\) and \\
          \(\boks{x,b-1}\in R_1 \implies \boks{x,b} \in R_1 \) (where \(x\neq a,a+1\)) and\\
          \(\boks{a+1,y}\in R_1 \implies \boks{a,y} \in R_1\) (where \(y\neq b-1,b\)).
        \item \((a,b+1) \in \sigma\) and \(\boks{a-1,b+1}\notin R_1\) and \\
          \(\boks{x,b+1}\in R_1 \implies \boks{x,b} \in R_1\) (where \(x\neq a-1,a\)) and\\
          \(\boks{a-1,y}\in R_1 \implies \boks{a,y} \in R_1\) (where \(y\neq b,b+1\)).
      \end{enumerate}
    \item \(\pi \preceq \sigma^{\boks{a,b}_d}\) and
      \begin{enumerate}
        \item \((a+1,b+1) \in \sigma\) and \(\boks{a+1,b+1}\notin R_1\) and \\
          \(\boks{x,b+1}\in R_1 \implies \boks{x,b} \in R_1\) (where \(x\neq a,a+1\)) and\\
          \(\boks{a+1,y}\in R_1 \implies \boks{a,y} \in R_1\) (where \(y\neq b,b+1\)).
        \item \((a,b) \in \sigma\) and \(\boks{a-1,b-1}\notin R_1\) and \\
          \(\boks{x,b-1}\in R_1 \implies \boks{x,b} \in R_1\) (where \(x\neq a-1,a\)) and\\
          \(\boks{a-1,y}\in R_1 \implies \boks{a,y} \in R_1\)  (where \(y\neq b-1,b\)).
      \end{enumerate}
    \end{enumerate}
\end{enumerate}
\end{proposition}
\begin{proof}
    By \NoteRef{not:downcmesh} we only need to show that an occurrence of
    \(m_1\) implies an occurrence of \(m_2\). We consider taking the first
    branch of every choice in condition~\eqref{prop:dom2:cond}. Consider a
    permutation \(\omega \in \av{\pi}\). Suppose \(\omega\) contains \(m_1\) and
    consider the region corresponding to \(\boks{a,b}\) in \(R_1\). If the
    region is empty, the occurrence of \(m_1\) is trivially an occurrence of
    \(m_2\).

    If the region is non-empty, then by \LemmaRef{lem:incdec} and
    condition~\eqref{prop:dom2:condc} of the proposition it must contain a
    decreasing subsequence. We can choose the topmost point in the region to
    replace the corresponding point in the mesh pattern and the points from the
    subsequence are now in the box southeast of the point.

    The conditions on the mesh ensure that no elements of the permutation that
    were inside a region corresponding to an unshaded box in the occurrence of
    \(m_1\) would be in a region corresponding to a shaded box in an occurrence
    of \(m_2\).

    Hence there are no points in the region corresponding to the box
    \(\boks{a,b}\) in the mesh pattern, and therefore we can shade this region.
    This implies that every occurrence of \(m_1\) in \(\av{\pi}\) can be
    modified into an occurrence of \(m_2\) so \(\av{\{\pi,m_2\}} \subseteq
    \av{\{\pi,m_1\}}\).

    Similar arguments cover the remainder of the branches.
\end{proof}

This proposition essentially states that we slide all of the points in the box
we desire to shade diagonally, and chose the topmost/bottommost point to replace
the original point in the mesh pattern. \FigureRef{fig:d2} show how the cases
apply to a general container of \(m_1\) to transform it into a container of
\(m_2\), the circled points are the same points in the permutation.
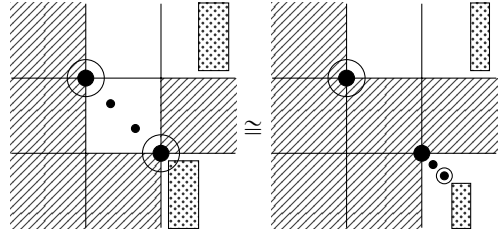
\begin{figure}
\begin{center}
\begin{tikzpicture}
	\modpattern[3]{1,2}{2,1}{0/0,0/1,0/2,1/0,2/1}
	\filldraw (1.33,1.66) circle (1.5pt);
	\filldraw (1.66,1.33) circle (1.5pt);
	\filldraw[pattern=crosshatch dots, pattern color=black!80] (2.1,0.9) rectangle (2.5,0.0);
	\filldraw[pattern=crosshatch dots, pattern color=black!80] (2.5,3) rectangle (2.9 ,2.1);
\end{tikzpicture}
\raisebox{8ex}{\(\coinc{}\)}
\begin{tikzpicture}
	\modpattern[3]{1}{2,1}{0/0,0/1,0/2,1/0,1/1,2/1}
	\filldraw (2.15,0.85) circle (1.5pt);
	\filldraw (2.3,0.7) circle (1.5pt);
	\draw (2.3,0.7) circle (3pt);
	\filldraw[pattern=crosshatch dots, pattern color=black!80] (2.4,0.6) rectangle (2.65,0.0);
	\filldraw[pattern=crosshatch dots, pattern color=black!80] (2.65,3) rectangle (2.9 ,2.1);
\end{tikzpicture}
\caption{A depiction of the second dominating pattern rule using our example patterns.}
\label{fig:d2}
\end{center}
\end{figure}

After implementing the Second Dominating Pattern Rule we find that there are
\(39\) \(\coinc[(2)]{231}\)-equivalence classes, of mesh patterns where the
underlying classical pattern is \(21\). By \NoteRef{note:main} there are
therefore exactly \(39\) \(\coinc{231}\)-equivalence classes.

\subsection{Coincidence classes of Av(\{231, (12, \textit{R})\}).}
When considering the coincidence classes of
\(\av{\{\perm{2,3,1},\mperm{1,2}{R}\}}\) we first apply the two Dominating
Pattern rules previously established. Starting from \(220\) classes, we result
in \(85 \coinc[(1)]{231}\)-equivalence classes reducing to \(59
\coinc[(2)]{231}\)-equivalence classes, considering mesh patterns with \(12\) as
the underlying classical patterns. However there are
\(56 \coinc[comp]{231}\)-equivalence classes.

For example the patterns
\begin{equation*}
    m_1 = \pattern{}{1,2}{0/2,2/0,2/1} \text{ and } m_2 = \pattern{}{1,2}{0/2,1/0,2/0,2/1}
\end{equation*}
are \(\coinc[comp]{231}\)-equivalent.

Consider an occurrence of \(m_1\) in a permutation, if the region corresponding
to the box \(\boks{1,0}\) is empty then we have an occurrence of \(m_2\). Now
look at the case when this region is not empty, and consider choosing the
rightmost point in the region. This gives us an occurrence of the following mesh
pattern.
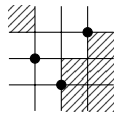
\begin{equation*}
    \pattern{}{2,1,3}{0/3,2/0,2/1,3/0,3/1,3/2}
\end{equation*}
By application of First Dominating Pattern Rule we obtain the following mesh
pattern
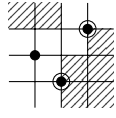
\begin{equation*}
    \pattern{2,3}{2,1,3}{0/3,1/3,2/0,2/1,3/0,3/1,3/2}
\end{equation*}
The highlighted points are an occurrence of the mesh pattern \(m_2\). This gives
rise to the following rule:

\begin{proposition}[Third Dominating Pattern Rule]
    \label{prop:dom3}
    Given two mesh patterns \(m_1 =(\sigma, R_1)\) and \(m_2 = (\sigma, R_2)\),
    and a dominating classical pattern \(\pi = (\pi,\emptyset)\), then
    \(m_1 \coinc{\pi} m_2\) if
    \begin{enumerate}
        \item The mesh set \(R_2 = R_1 \cup \{\boks{a,b}\}\)
        \item\label{prop:dom3:condocc} One of the patterns in \(m_1^\boksall{a,b}\)
            is coincident with a mesh pattern containing an occurrence of
            \(m_2\) as a subpattern.
    \end{enumerate}
\end{proposition}
\begin{proof}
    \NoteRef{not:downcmesh} implies \(\av{\{\pi,m_1\}} \subseteq \av{\{\pi,m_2\}}\)
as before. Now consider a permutation \(\omega \in \av{\pi}\) that contains an
occurrence of \(m_1\). If the region corresponding to the box \(\boks{a,b}\) is
empty then we have an occurrence of \(m_2\). By
condition~\ref{prop:dom3:condocc} of the proposition there exists a pattern
\(r\) in \(m_1^\boksall{a,b}\) such that \(m_2\) is a subpattern in some pattern
\(r'\) coincident with \(r\). If the region is non-empty then there exists an
occurrence of \(r\) in \(\omega\). Therefore, there is an occurrence of \(r'\) in
\(\omega\), which implies that there is an occurrence of \(m_2\) in \(\omega\).
Thus \(\av{\{\pi,m_2\}} \subseteq \av{\{\pi,m_1\}}\) and the two patterns are
coincident.
\end{proof}

After implementing the Third Dominating Pattern Rule we find that there are
\(56\) \(\coinc[(3)]{231}\)-equivalence classes, of mesh patterns where the
underlying classical pattern is \(12\). By \NoteRef{note:main} there are
therefore exactly \(56\) \(\coinc{231}\)-equivalence classes, since we have
observed \(56\) \(\coinc[comp]{231}\)-equivalence classes through
experimentation.

\subsection{Coincidence classes of Av(\{321, (12, \textit{R})\}).}
When considering coincidences of mesh patterns with underlying classical pattern
\(\perm{1,2}\) in \(\av{\perm{3,2,1}}\) application of the previously
established rules gives no coincidences. Through experimentation we discover
that there are \(7\) \(\coinc[comp]{321}\)-equivalence classes that are
unexplained. Since the number of coincidences is so small we will reason for
these coincidences without attempting to generalise into concrete rules.

Intuitively, it is easy to see why our previous rules have no power here. It is
impossible to add a single point to a mesh pattern \((12, R)\) and create an
occurrence of \(\pi = \perm{3,2,1}\). It is also impossible to have a position
where addition of an ascent, or descent, provides extra shading power.

The patterns
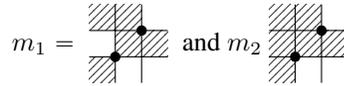
\begin{equation*}
    m_1 = \pattern{}{1,2}{0/0, 0/2, 1/1, 1/2, 2/1} \et m_2 \pattern{}{1,2}{0/0,
    0/1, 0/2, 1/1, 1/2, 2/1}
\end{equation*}
are \(\coinc[comp]{321}\)-equivalent. (There are 3 symmetries of these
patterns that are also equivalent to each other by the same reasoning.)

Consider the region corresponding to the box \(\boks{0,1}\) in any occurrence of
\(m_1\) in a permutation. By \LemmaRef{lem:incdec} it must contain an increasing
subsequence. If the region is empty then we have an occurrence of \(m_2\). If
there is only one point in the region we can choose this point to replace the
\(1\) in the mesh pattern to get the required shading. If there is more than one
point then choosing the two leftmost points gives us the following mesh pattern.
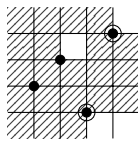
\begin{equation*}
    \pattern{3,4}{2,3,1,4}{0/0, 0/1, 0/2, 0/3, 0/4,
                            1/0, 1/1, 1/2, 1/3, 1/4,
                            2/0, 2/1,2/2, 2/4,
                            3/1, 3/2, 3/3, 3/4,
                            4/1, 4/2, 4/3}
\end{equation*}
Here the two highlighted points are the original two points. The other two
points are an occurrence of the pattern we \(m_2\), and hence the two patterns
are coincident. It is also possible to calculate this coincidence by an
extension of the Third Dominating rule, where we allow a sequence of point
addition operations, this is discussed further in the future work section.
Because of symmetries we have dealt with \(4\) out of the \(7\)
\(\coinc[comp]{321}\)-equivalence classes.

Now consider the patterns
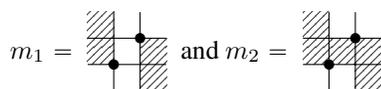
\begin{equation*}
    m_1 = \pattern{}{1,2}{0/1,0/2,2/0,2/1} \text{ and } m_2 = \pattern{}{1,2}{0/1,0/2,1/1,2/0,2/1}
\end{equation*}
which are \(\coinc[comp]{321}\)-equivalent.

In order to prove this coincidence we will proceed by mathematical
induction on the number of points in the region corresponding to the middle box.
We call this number \(n\).
\begin{description}
    \item [Base Case \((n=0)\)] The base case holds since we can freely shade the
        box if it contains no points.
    \item [Inductive Hypothesis \((n=k)\)] Suppose that we can find an occurrence
        of the second pattern if we have an occurrence of the first with \(k\) points
        in the middle box.
    \item [Inductive Step \((n=k+1)\)] Suppose that we have \((k+1)\) points in
        the middle box. Choose the bottom most point in the middle box, giving the
        mesh pattern
     
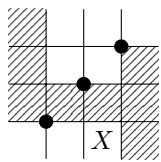
\begin{equation*}
        \begin{tikzpicture}[scale=0.5]
            \modpattern[5]{}{1,2,3}{0/1,0/2,0/3,1/1,2/1,3/0,3/1,3/2}
            \draw (2.5,0.5) node {\(X\)};
        \end{tikzpicture}
    \end{equation*}
        Now we need to consider the box labelled \(X\). If this box is empty
        then we have an occurrence of \(m_2\) and are done. If this box contains
        any points then we gain some extra shading on the mesh pattern as any
        points in the boxes \(\boks{1,2}\) and \(\boks{1,3}\) would create an
        occurrence of the dominating pattern \(\perm{3,2,1}\)
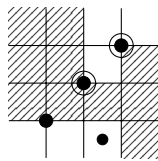
     \begin{equation*}
        \begin{tikzpicture}[scale=0.5]
            \modpattern[5]{2,3}{1,2,3}{0/1,0/2,0/3,1/1,1/2,1/3,2/1,3/0,3/1,3/2}
            \filldraw (2.5,0.5) circle (4 pt);
        \end{tikzpicture}
    \end{equation*}
        The two highlighted points form an occurrence of \(m_1\) with \(k\)
        points in the middle box, and thus by the Inductive Hypothesis we
        are done.
\end{description}
By induction we have that every occurrence of \(m_1\) leads to an occurrence of \(m_2\)
and by \NoteRef{not:downcmesh} every occurrence of \(m_2\) is an occurrence of
\(m_1\) so the two patterns are coincident.
This argument applies to another two pairs of classes. Therefore we have
explained all 7 of the coincidences in \(\av{\{\perm{3,2,1}, \mperm{1,2}{R}\}}\)
and there are 213 \(\coinc{321}\)-equivalence classes.

\section{Wilf Equivalences between equivalence classes}
Wilf-equivalence is an important aspect to study in the field of permutation
patterns. The original definition is as follows:

\begin{definition}
  Two patterns \(\pi\) and \(\sigma\) are said to be \emph{Wilf-equivalent}
  if for all \(k\ge0, \size{\av[k]{\pi}} = \size{\av[k]{\sigma}}\). Two sets of
  permutation patterns \(R\) and \(S\) are \emph{Wilf-equivalent} if for all
  \(k\ge0, \size{\av[k]{R}} = \size{\av[k]{S}}\).
\end{definition}

Coincident patterns are trivially Wilf-equivalent: if \(\av[k]{R} = \av[k]{S}\)
then trivially \(\size{\av[k]{R}} = \size{\av[k]{S}}\). Coincidence is
therefore a stronger equivalence condition than Wilf-equivalence.

When examining Wilf-equivalences we can use a number of symmetries to reduce the
amount of work required. It can be seen that the reverse, complement and inverse
operations (see \FigureRef{fig:symm}) preserve enumeration, and therefore
classes related by these symmetries are trivially Wilf-equivalent.
\begin{figure}[!htb]
\begin{align*}
    \rev{\textpattern{}{2,3,1}{}} &= \textpattern{}{1,3,2}{}\\
    \com{\textpattern{}{2,3,1}{}} &= \textpattern{}{2,1,3}{}\\
    \inv{\textpattern{}{2,3,1}{}} &= \textpattern{}{3,1,2}{}
\end{align*}
\caption{The operations reverse, complement and inverse for the pattern 231}
\label{fig:symm}
\end{figure}

Since we will consider Wilf-equivalences in a set \(\av{S}\) we
must only use symmetries that preserve the dominating pattern(s) in \(S\), if we
were to allow other symmetries, then the equivalences calculated in the
previous section do not necessarily hold.

In the remainder of the paper we will consider Wilf-equivalences of patterns
whilst avoiding the \emph{dominating pattern} \(\perm{2,3,1}\), if two patterns
\(p_1 \et p_2\) are Wilf-equivalent we will denote this \(p_1\weqv{231}p_2\),
similarly when computational methods indicate that two patterns are Wilf
equivalent under \(231\) we write \(p_1\weqv[comp]{231}p_2\). We will use
\(\mathcal{C}\) to denote \(\av{\perm{2,3,1}}\) and \(C(x)\) will be the usual
Catalan generating function satisfying \(C(x) = 1 + xC(x)^2\). The fact that
\(C(x)\) is the generating function for \(\mathcal{C}\) can be seen by
structural decomposition around the maximum, as
shown in \FigureRef{fig:decompmax}.

\begin{figure}[!ht]
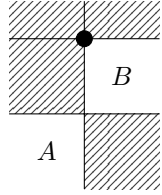

    \centering
    \decompmax{0/1,1/0}{A}{}{}{B}
    \caption{Structural decomposition of a non-empty avoider of 231}
    \label{fig:decompmax}
\end{figure}

The elements to the left of the maximum, \(A\), have the structure of a \(231\)
avoiding permutation, and the elements to the right of the maximum, \(B\), also have
the structure of a \(\perm{2,3,1}\) avoiding permutation. Furthermore, all the
elements in \(A\) lie below all of the elements in \(B\). We call \(A\) the
\emph{lower-left section} and \(B\) the \emph{upper-right section}.

We can also decompose a permutation avoiding \(\perm{2,3,1}\) around the
leftmost point, giving a similar figure.

\subsection{Wilf-classes with mesh patterns of length 1.}
\label{wilf1}
When considering the mesh patterns of length \(2\) it will be useful to know the
Wilf-equivalence classes of the mesh patterns of length \(1\) inside
\(\av{\perm{2,3,1}}\), this means that we are considering the
\(\weqv{231}\)-equivalence classes of mesh patterns with underlying classical
pattern \(1\).

The patterns in the following set are coincident,
\begin{equation*}
    \left\{
        \begin{array}{c}
        \textpattern{}{1}{},
        \textpattern{}{1}{0/0},
        \textpattern{}{1}{1/0},
        \textpattern{}{1}{0/0,1/0},
        \textpattern{}{1}{0/1},\\
        \textpattern{}{1}{0/0,0/1},
        \textpattern{}{1}{1/1},
        \textpattern{}{1}{1/0,1/1},
        \textpattern{}{1}{0/1,1/1}
    \end{array}
    \right\}
\end{equation*}
due to the fact that every permutation, except the empty
permutation, must contain an occurrence of all of these patterns.

The pattern \(\textpattern{}{1}{0/0,1/0,0/1,1/1}\) is in its own
\(\weqv{231}\)-equivalence class since the only permutation containing it is
the permutation \(\perm{1}\). The avoiders of this pattern therefore
have generating function \(\var(x) = C(x)-x\).

\nextvar
The pattern \(p = \textpattern{}{1}{0/1,1/0}\) is one of the
quadrant marked mesh patterns studied by \textcite{2012arXiv1201.6243K}.
Alternatively we can enumerate
avoiders of \(p\) by decomposing a non-empty avoider of \(p\) around the
maximum element
in order to give the following structural decomposition.

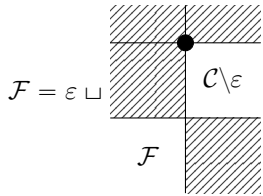
\begin{equation*}
    \scriptvar = \varepsilon \sqcup
\decompmax{0/1,1/0}{\scriptvar}{}{}{\mathcal{C}\setminus\varepsilon}
\end{equation*}
If the upper-right section was empty the maximum would create an
occurrence of the pattern, however no points in this section can create
an occurrence since the maximum lies in a region corresponding to the
shading in \(p\), so we can use any avoider of \(\perm{2,3,1}\). The lower-left
section however can create occurrences
of \(p\) and therefore must also avoid \(p\), as well as \(\perm{2,3,1}\).
This gives the generating function of avoiders to be the function \(\var(x)\)
satisfying
\begin{align*}
    \var(x) &= 1 + x\var(x)(C(x)-1) \\
    \shortintertext{Solving for \(\var\) gives}
    \var(x) &= \frac{1}{1-x(C(x)-1)}
\end{align*}

\noindent
Calculating coefficients given by this generating function gives the Fine
numbers.
\begin{equation*}
    1, 0, 1, 2, 6, 18, 57, 186, 622, 2120, 7338,\dotsc
\tag{\href{https://oeis.org/A000957}{OEIS: A000957}}
\end{equation*}

\nextvar[\varfine]
It can be shown by use of the Second Dominating Pattern Rule that the patterns
\textpattern{}{1}{0/0,1/1} and \(q_1 = \textpattern{}{1}{0/0,0/1,1/1}\) are
coincident under \(\perm{2,3,1}\). Consider the decomposition of a non-empty
avoider of \(q_1\) in \(\av{\perm{2,3,1}}\) around the maximum:
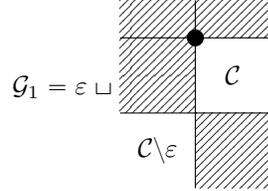
\begin{equation*}
    \scriptvar_1 =\varepsilon \sqcup
\decompmax{0/1,1/0}{\mathcal{C}\setminus\varepsilon}{}{}{\mathcal{C}}
\end{equation*}
This can be explained succinctly by the fact that a permutation contains
\(q_1\) if and only if it starts with its maximum, and by not allowing the
lower-left section of the
\(\perm{2,3,1}\) avoider to be empty we prevent an occurrence from ever
happening.

Consider \(q_2 = \textpattern{}{1}{0/1,1/0,1/1}\). Avoiding this pattern
means that a permutation does not end with its maximum. We can perform a
similar
decomposition as before to get
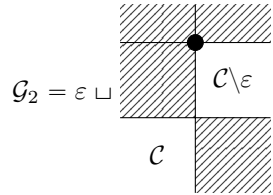
\begin{equation*}
    \mathcal{\var}_2 =\varepsilon \sqcup
\decompmax{0/1,1/0}{\mathcal{C}}{}{}{\mathcal{C}\setminus\varepsilon}
\end{equation*}
The pattern \(q_3 = \textpattern{}{1}{0/0,0/1,1/0}\) is the
reverse-complement-inverse of \(q_2\) and hence the avoiders of \(q_2\) and
\(q_3\) (\(\scriptvar_2\) and \(\scriptvar_3\)) are equinumerous, and so \(q_1
\weqv{231} q_2 \weqv{231}q_3\). All of these classes have the same generating
function, namely
\begin{equation}
    \var(x) = 1 + xC(x)(C(x)-1). \label{eqn:maxlgen}
\end{equation}
The coefficients of this generating function are
\begin{equation*}
    1, 0, 1, 3, 9, 28, 90, 297, 1001, 3432, 11934,\dotsc
\tag{\href{https://oeis.org/A000245}{OEIS: A000245} with offset \(1\)}
\end{equation*}

\nextvar[\varmaxl]
There is one pattern of length \(1\) still to consider. The pattern
\(r = \textpattern{}{1}{0/0,1/0,1/1}\) is avoided by all permutations
that do not end in their minimum. Any avoider of \(\perm{2,3,1}\) that
ends in its minimum must be a decreasing sequence.
Therefore this particular class has equation
\begin{equation*}
    \var(0) = 1, \var(x) = C(x)-1
\end{equation*}
Computing these values gives
\begin{equation*}
    	1, 0, 1, 4, 13, 41, 131, 428, 1429, 4861, 16795,\dotsc
\tag{\href{https://oeis.org/A141364}{OEIS: A141364}}
\end{equation*}

\section{Wilf-classes with patterns of length 2}
From \SectionRef{sec:coincs} we know there are \(95\) coincidence classes and
therefore at most that many Wilf-equivalence classes.

The only symmetry that we are able to consider is
\emph{reverse-complement-inverse} as this is the only symmetry that preserves
the \(\perm{2,3,1}\) pattern. Using this symmetry we can find 61 classes of
trivial Wilf-equivalence, these equivalences being explained by either the
patterns being coincident in \(\av{\perm{2,3,1}}\), or by one pattern being the
reverse-complement-inverse of some other pattern.

Computing avoiders up to length 10 gives 23 \(\weqv[comp]{231}\)-equivalence
classes, of which 13 are not equal to coincidence classes (after taking
symmetries). Therefore there seem to be Wilf-equivalences that are not explained
by coincidences or symmetry.

We will use two main methods of establishing these Wilf-equivalences: the
structural decomposition of avoiders, via generating functions; or the structure
of a general permutation containing the pattern, looking at a particular
occurrence of the pattern in a permutation avoiding \(\perm{2,3,1}\). Sometimes
it will be necessary to use both of these methods to consolidate a single
Wilf-class.

\subsection{}
\nextvar
The following patterns are in the same \(\weqv[comp]{231}\)-equivalence class
\begin{gather}
    m_1 = \pattern{}{1,2}{0/0,0/1,1/1,1/2,2/0,2/1},
    m_2 = \pattern{}{2,1}{0/0,0/1,1/1,1/2,2/0,2/1}, \label{eq:singleud}\\
    m_3 = \pattern{}{2,1}{0/1,0/2,1/0,1/1,1/2,2/1}, \text{ and }
    m_4 = \pattern{}{2,1}{0/1,0/2,1/0,1/1,1/2,2/0} \label{eq:other}
\end{gather}

First we prove that \(m_1 \weqv{231} m_2\) by considering the form of a general permutation containing
either of the two patterns. First looking at a general occurrence of \(m_1\) in a permutation
in \(\av{231}\)
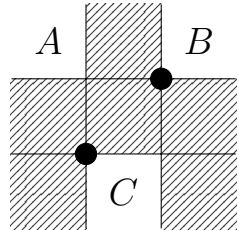
\begin{equation*}
\begin{tikzpicture}[scale=1.0]
\modpattern[4]{}{1,2}{0/0,0/1,1/1,1/2,2/0,2/1}
\node at (0.5,2.5) {\Large \(A\)};
\node at (2.5,2.5) {\Large \(B\)};
\node at (1.5,0.5) {\Large \(C\)};
\end{tikzpicture}
\end{equation*}

If there is an occurrence of \(m_1\) there must be an occurrence of \(m_1\) that minimizes
the values of the points \(xy\) in the occurrence, first with respect to \(x\) then with respect to \(y\).
Now consider the top regions, labelled \(A \et B\), the subpermutation contained
in the union of these regions must avoid the permutation \(231\). Also, the
subpermutation in the region \(A\) must be a decreasing subsequence, otherwise
an occurrence of \(231\) will be created with either of the points in the
occurrence of \(m_1\). Now, consider the region labelled \(C\), since we
specified that we were focused on the lowest possible occurrence of \(m_1\) this
region cannot contain an occurrence of either \(m_1\) or \(231\). This is a full
structural decomposition of a container of \(m_1\) inside \(\av{231}\).

Now consider a general occurrence of \(m_2\) in a permutation
in \(\av{231}\)
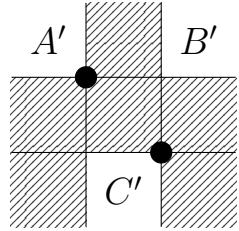
\begin{equation*}
\begin{tikzpicture}[scale=1.0]
\modpattern[4]{}{2,1}{0/0,0/1,1/1,1/2,2/0,2/1}
\node at (0.5,2.5) {\Large \(A^\prime\)};
\node at (2.5,2.5) {\Large \(B^\prime\)};
\node at (1.5,0.5) {\Large \(C^\prime\)};
\end{tikzpicture}
\end{equation*}

Similarly to a container of \(m_1\) we will look at the lowest possible
occurrence of \(m_2\), and as before the regions \(A^\prime \et B^\prime\)
together contain an avoider of \(231\) and \(A^\prime\) contains a decreasing
subsequence. Since we are considering the lowest occurrence of \(m_2\) the
region \(C^\prime\) does not contain an occurrence of either \(m_2\) or \(231\).

Since all the regions in both of these cases contain the same parts,  the
classes defined by containment of \(m_1\) and \(m_2\) inside
\(\av{\perm{2,3,1}}\) are equinumerous and therefore so are their avoiders.

Consider the class, \(\scriptvar_3\), of permutations defined by avoiding
\(\perm{2,3,1}\) and \(m_3\). We can decompose a member of this class around the
maximum
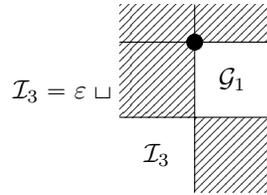
\begin{equation*}
    \scriptvar_3 = \varepsilon \sqcup
    \decompmax{1/0,0/1}{\scriptvar_3}{}{}{\mathcal{\varmaxl}_1}
\end{equation*}
Only the first point in the top right region can create an occurrence of \(m_3\)
if and only if it is the element with largest value in this region, therefore
the partial permutation in this region must avoid starting with the maximum,
\ie, be in \(\mathcal{\varmaxl}_1 = \av{\textpattern{}{1}{0/0,0/1,1/1}}\)
described
in \SectionRef{wilf1}.

Looking at avoiders of \(\perm{2,3,1}\) and \(m_4\) we
can perform a similar decomposition around the maximum to get
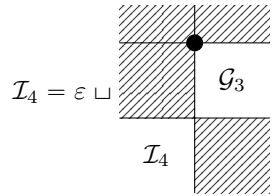
\begin{equation*}
    \scriptvar_4 = \varepsilon \sqcup
    \decompmax{1/0,0/1}{\scriptvar_4}{}{}{\mathcal{\varmaxl}_3}
\end{equation*}
An occurrence of \(m_4\) can never occur entirely within the
top right region, since the maximum is to the left and above any point in this region. It could only occur between the maximum and the first point
in the region, if and only if this first point is the lowest valued element in
this region, so this top right region must contain a sub-permutation that does not
start with it's minimum, i.e., it is a member of \(\mathcal{\varmaxl}_3 = \av{\textpattern{}{1}{0/0,0/1,1/0}}\).
Since both \(\mathcal{\varmaxl}_1\) and \(\mathcal{\varmaxl}_3\) have the same
enumeration, \(\scriptvar_3\) and \(\scriptvar_4\) must also have
the same enumeration and are therefore \(m_3\weqv{231}m_4\).

Now we must consolidate these two subclasses. In order to do this we must
consider the decomposition around the leftmost point of a permutation in
\(\av{\{\perm{2,3,1},m_1\}}\). We have the following
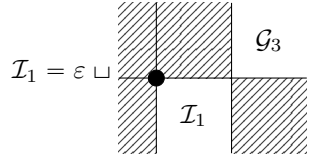
\begin{equation*}
    \scriptvar_1 = \varepsilon \sqcup
    \decompleft{1/0,0/1}{\scriptvar_1}{}{}{\mathcal{\varmaxl}_3}
\end{equation*}
It is therefore clear that avoiders of \(m_1\) and avoiders of \(m_4\) have the
same enumeration, and therefore \(m_1 \weqv{231} m_2 \weqv{231}m_3\weqv{m_4}\)
with generating function satisfying
\begin{equation*}
    \var(x) = 1 + x\var(x)\varmaxl(x)
\end{equation*}
where \(\varmaxl(x)\) is the generating function given in \EquationRef{eqn:maxlgen}.
This can be enumerated to give the sequence
\begin{equation*}
    1, 1, 1, 2, 6, 19, 61, 200, 670, 2286, 7918,\dotsc \tag{\href{https://oeis.org/A035929}{OEIS: A035929} offset 1}
\end{equation*}

\subsection{}
\nextvar
The patterns \(m_1 \et m_2\) are in the same \(\weqv[comp]{231}\)-equivalence class
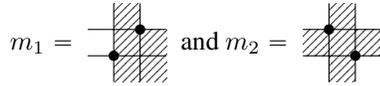
\begin{equation*}
    m_1 = \pattern{}{1,2}{1/0,1/1,1/2,2/0,2/1} \text{ and }
    m_2 = \pattern{}{2,1}{0/1,1/0,1/1,1/2,2/1}
\end{equation*}
Let \(\scriptvar_1\) be the set of avoiders of \(m_1\) in \(\av{231}\). By
structural decomposition around the leftmost point we have
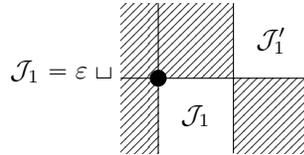
\begin{equation*}
    \scriptvar_1 = \varepsilon \sqcup
    \decompleft{0/1,1/0}{\scriptvar_1}{}{}{\scriptvar^\prime_1}
\end{equation*}

\noindent
Here \(\scriptvar^\prime_1\) is a permutation avoiding \(\perm{2,3,1}, m_1\) and
\(\textpattern{}{1}{0/0,0/1,1/0}\). Now consider the decomposition of a
permutation in \(\scriptvar^\prime_1\). It can once again be decomposed around
the leftmost point
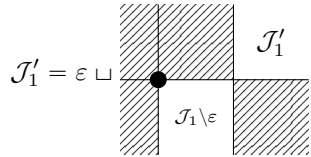
\begin{equation*}
    \scriptvar^\prime_1 = \varepsilon \sqcup
    \decompleft{0/1,1/0}{\scriptstyle \scriptvar_1\setminus\varepsilon}{}{}{\scriptvar^\prime_1}
\end{equation*}
This is a complete decomposition of avoiders of \(m_1\). Now we look at an
avoider of \(m_2\), decomposed around the leftmost point
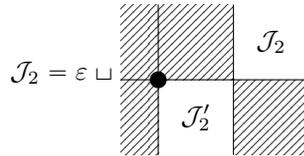
\begin{equation*}
    \scriptvar_2 = \varepsilon \sqcup
    \decompleft{0/1,1/0}{\scriptvar^\prime_2}{}{}{\scriptvar_2}
\end{equation*}

Where \(\scriptvar^\prime_2\) is a permutation avoiding \(\perm{2,3,1}, m_2\)
and \(\textpattern{}{1}{0/0,0/1,1/1}\). Again we use the same method of
decomposition of a permutation in \(\scriptvar^\prime_2\)
\begin{equation*}
    \scriptvar^\prime_2 = \varepsilon \sqcup
    \decompleft{0/1,1/0}{\scriptvar^\prime_2}{}{}{\scriptstyle \scriptvar_2\setminus\varepsilon}
\end{equation*}
This gives us a generating function \(\var(x)\) satisfying
\begin{align}
    \var(x) &= 1 + x\var(x)\var^\prime(x) \label{eqn:Pgen}\\
    \var^\prime(x) &= 1 + x(\var(x)-1)\var^\prime(x)\label{eqn:PprimeGen}
\end{align}
Solving \EquationRef{eqn:PprimeGen} for \(\var^\prime(x)\) and substituting into
\EquationRef{eqn:Pgen} gives us that the generating function for
\(\var(x)\) satisfies
\begin{equation}
    \var(x) = x\var^2(x) - x(\var(x) - 1) + 1
\end{equation}
Evaluating \(\var(x)\) gives us the sequence
\begin{equation*}
    1, 1, 1, 2, 4, 9, 21, 51, 127, 323, 835,\dotsc \tag{\href{https://oeis.org/A001006}{OEIS: A001006 with offset 1}}
\end{equation*}
Which is an offset of the Motzkin numbers.

In order to establish the remainder of the Wilf-equivalences of the form
\(\av{\{\perm{2,3,1},p\}}\) where \(p\) is a mesh pattern of length \(2\) we can
use similar methods to allow us to consolidate experimental classes into actual
classes, these methods allow us to explain all 23 of the observed Wilf-classes
seen in experimentation.

\section{Conclusions and Future Work}
If we consider a similar approach to dominating patterns of length \(4\) and
mesh patterns of length \(2\), it can be seen that the number of cases required
to establish rules increases to a number that is infeasible to compute manually.
For an extension of the First Dominating rule alone, we would have to consider
placement of points in any pair of unshaded regions. The fact that the rules
established do not completely cover the coincidences with a dominating pattern
of length \(3\) shows that this is a difficult task.

It is interesting to consider the application of the Third Dominating rule, as
well as the simple extension of allowing a sequence of point insertions, to mesh
patterns without any dominating pattern in order to try to capture some of the
coincidences described in
\textcite{DBLP:journals/combinatorics/HilmarssonJSVU15} and
\textcite{DBLP:journals/corr/ClaessonTU14}.
 \begin{example}
   The coincidence of the patterns
   \begin{equation*}
     m_1 = \pattern{}{1,2}{0/1,0/2,1/1,1/2,2/0}, \text{ and } m_2 = \pattern{}{1,2}{0/2,1/0,1/1,2/0,2/1}
   \end{equation*}
	does not follow from the general methods presented by
	\textcite{DBLP:journals/corr/ClaessonTU14}, but is rather handled there as a
	special case. We can do it as follows:
   Consider a permutation containing \(m_1\),
   
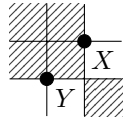
\begin{equation*}
     \begin{tikzpicture}[scale=0.5]
         \modpattern[5]{}{1,2}{0/1,0/2,1/1,1/2,2/0}
         \draw (1.5,0.5) node {\(Y\)};
         \draw (2.5,1.5) node {\(X\)};
     \end{tikzpicture}
   \end{equation*}
   If the regions corresponding to both \(X\) and \(Y\) are empty then we have
   an occurrence of \(m_2\).
   If the region corresponding to \(X\) is non-empty, we can then choose
   the lowest valued point in this region
   
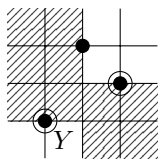
\begin{equation*}
     \begin{tikzpicture}[scale=0.5]
         \modpattern[5]{1,3}{1,3,2}{0/1,0/2,0/3,1/1,1/2,1/3,2/0,2/1,3/0,3/1}
         \draw (1.5,0.5) node {\(Y\)};
     \end{tikzpicture}
   \end{equation*}
   If the region corresponding to \(Y\) is empty then we have an occurrence of
   \(m_2\) with the indicated points.
   Now if the region corresponding to \(Y\) is non-empty, we can choose the
   rightmost point in this region.
   
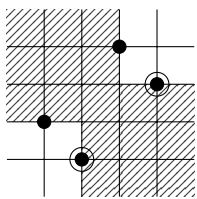
\begin{equation*}
     \begin{tikzpicture}[scale=0.5]
         \modpattern[5]{2,4}{2,1,4,3}{0/2,0/3,0/4,1/2,1/3,1/4,2/0,2/1,2/2,2/3,2/4,3/0,3/1,3/2,4/0,4/1,4/2}
     \end{tikzpicture}
   \end{equation*}
   And now the two indicated points form an occurrence of \(m_2\).
   We have therefore shown that any occurrence of \(m_1\) leads to an occurrence of
   \(m_2\) and we can easily show the converse by the same reasoning, so \(m_1\)
   and \(m_2\) are coincident.
   This is captured by an extension of the Third Dominating rule where we allow
   multiple steps of adding points before we check for subpattern containment.
 \end{example}

It would be interesting to consider a systematic explanation of
Wilf-equivalences among classes where \(\perm{3,2,1}\) is the dominating
pattern, possibly using the  construction presented in
\textcite[Sec.~11]{2015arXiv151203226B}, in order to directly reach enumeration
and hopefully establish some of the non-trivial Wilf-equivalences between
classes with different dominating patterns.

\printbibliography
\end{document}